\documentclass{amsart}
\usepackage{latexsym,amssymb,amsthm,amsmath,amscd}

\theoremstyle{plain}
\newtheorem{theorem}{Theorem}[section]
\newtheorem*{Theorem B}{Theorem B}
\newtheorem*{Theorem A}{Theorem A}

\newtheorem{proposition}{Proposition}[section]

\newtheorem{corollary}{Corollary}[section]

\newtheorem{example}{Example}[section]
\numberwithin{equation}{section}

\theoremstyle{remark}

 \numberwithin{equation}{section}

\def\<{\left < }
\def\>{\right >}
\def\({\left ( }
\def\){\right )}

\def\e{\eqref}

\def\E4{$\mathbb E^4$}
\def\cn{\hskip.02in{\rm cn}}
\def\sn{\hskip.02in{\rm sn}}
\def\dn{\hskip.02in{\rm dn}}

\def\sd{\hskip.02in{\rm sd}}
\def\cd{\hskip.02in{\rm cd}}
\def\nd{\hskip.02in{\rm nd}}
\def\ns{\hskip.02in{\rm ns}}
\def\ds{\hskip.02in{\rm ds}}

\begin{document}

\markboth{B.-Y. Chen}{Ideal hypersurfaces in $\mathbb E^4$}

\title[Ideal hypersurfaces in $\mathbb E^4$]{Ideal hypersurfaces of Euclidean four-space}

\author[B.-Y. Chen]{Bang-Yen Chen}
 \address{Department of Mathematics\\Michigan State University\\619 Red Cedar Road \\ East Lansing, MI 48824--1027, USA}
\email{bychen@math.msu.edu}

\begin{abstract} The notion of ideal immersions was introduced by the author in 1990s. Roughly speaking, an ideal immersion of a Riemannian manifold into a real space form is a nice isometric immersion which produces the least possible amount of tension from the ambient space at each point. 

In this paper, we classify all ideal hypersurfaces with two distinct principal curvatures in the Euclidean 4-space $\mathbb E^4$. Moreover, we prove that such ideal hypersurfaces are always rigid. Furthermore, we show that non-minimal ideal hypersurfaces with three distinct principal curvatures in $\mathbb E^4$ are also rigid.
On the other hand, we provide explicit examples to illustrate that minimal ideal hypersurfaces with three principal curvatures in $\mathbb E^4$ are not necessary rigid.
\end{abstract}

\keywords{Ideal immersion, ideal hypersurface, $\delta$-invariants, Chen invariants, rigidity, fundamental inequalities.}

 \subjclass[2000]{53C40, 53C42}

\maketitle

\section{Introduction}

For a Riemannian manifold $M$ with $n=\dim M\geq 3$, the author introduced in early 1990s a Riemannian invariant $\delta_M$ defined by \cite{c93}
\begin{align} \label{1.1} \delta_M(p)=\tau (p)-\inf K(p),\end{align}
where $\tau$ is the scalar curvature of $M$ and $\inf K(p)$ is the function assigning to the point $p$ the infimum of the sectional curvature $K(\pi)$, running  over all 2-planes in $T_pM$.

For an isometric immersion of a Riemannian $n$-manifold $M$ into an $m$-dimensional Riemannian space form $R^m(\epsilon)$ of
constant sectional curvature $\epsilon$, the author proved in \cite{c93} the following sharp inequality:
\begin{align}\label{1.2} \delta_M\leq \frac{n^2(n-2)}{2(n-1)} H^2 +\frac{1}{2}(n+1)(n-2)\epsilon,\end{align}
involving the $\delta$-invariant $\delta_M$ and the squared mean curvature $H^2$.

Inequality \e{1.2} has many important applications, for example, it provides a Riemannian obstruction for a Riemannian manifold to admit a minimal
isometric immersion into a Euclidean space. It also gives rise to an obstruction to
Lagrangian isometric immersions from compact Riemannian manifolds with finite fundamental group into complex space forms. The invariant $\delta_M$ and the inequality \e{1.2} were later extended by the author to the general $\delta$-invariants $\delta(n_1,\ldots,n_k)$ (also known as Chen invariants) and general inequalities involving $\delta(n_1,\ldots,n_k)$ (see \cite{c98,c00,c2000,c08,book} for more details).

Since \e{1.2} is a very general and sharp inequality, it is very natural and interesting to investigate submanifolds satisfying the equality case of inequality \e{1.2} identically. Following \cite{c00,book}, we call a submanifold satisfying the equality case of \e{1.2} identically a $\delta(2)$-ideal submanifold. 

In this paper, we classify all ideal hypersurfaces with two distinct principal curvatures in the Euclidean 4-space $\mathbb E^4$. Moreover, we prove that such ideal hypersurfaces in $\mathbb E^4$ are always rigid. Furthermore, we show that non-minimal ideal hypersurfaces with three distinct principal curvatures are also rigid.
On the other hand, we provide explicit examples to show that minimal ideal hypersurfaces with three principal curvatures in $\mathbb E^4$ are not necessary rigid.

\section{Preliminaries}
\subsection{Basic formulas}
Let $M$ be a Riemannian $n$-manifold equipped with an inner product $\<\;\, ,\;\>$. Denote by $\nabla$ the Levi-Civita connection of $M$. 

Assume that $M$ is isometrically immersed in a Euclidean $m$-space $\mathbb E^m$. Then the formulas of Gauss and Weingarten  are given respectively by (cf. \cite{c73,book})
\begin{align}\label{2.1} &\tilde \nabla_XY=\nabla_X Y+h(X,Y),\\&\label{2.2} \tilde
\nabla_X\xi =-A_\xi X+D_X\xi,\end{align} for vector fields $X$ and $Y$ tangent to $N$ and  $\xi$ normal to $N$, where $\tilde \nabla$ denotes the Levi-Civita connection on $\mathbb E^m$, $h$  is the second fundamental form, $D$ is the normal connection, and $A$ is the  shape operator of $N$. 

The second fundamental form $h$ and the shape operator $A$ are related by
\begin{align}\label{2.3}\<A_\xi X,Y\>=\<h(X,Y), \xi\>,\end{align} where $\<\;\, ,\;\>$
is the inner product on $N$ as well as on $\tilde M$. The mean curvature vector of $N$ is defined by
\begin{align}\label{2.4} \overrightarrow{H}=\frac{1}{n} \,{\rm trace} \,h,\;\; n=\dim N.\end{align}
The squared mean curvature $H^{2}$ is given by $H^{2}=\<\right.\hskip-.02in\overrightarrow{H},\overrightarrow{H} \hskip-.02in\left.\>$.

The  {\it equation of Gauss\/} is given by
\begin{equation} \label{2.5} R(X,Y;Z,W)=\<h(X,W),h(Y,Z)\>-\<h(X,Z),h(Y,W)\>\end{equation} for vectors
$X,Y,Z,W$ tangent to $M$, where $R$ denotes the Riemann curvature tensors of $M$.

For the second fundamental form $h$, we define its covariant derivative ${\bar\nabla}h$ with respect to the connection on
$TM \oplus T^{\perp}M$ by
\begin{align}\label{2.6}({\bar\nabla}_{X}h)(Y,Z)=D_{X}(h(Y,Z))-h(\nabla_{X}Y,Z) -h(Y,\nabla_{X}Z).\end{align}

The {\it equation of Codazzi\/} is
\begin{align}\label{2.7}({\bar\nabla}_{X}\sigma)(Y,Z)=
({\bar\nabla}_{Y}\sigma)(X,Z),\end{align} for vectors $X,Y,Z$ tangent to $M$.

\subsection{$\delta$-invariants}
 Let $M$ be a  Riemannian $n$-manifold. Let $K(\pi)$ denote the sectional curvature of $M$ associated with a plane section $\pi\subset T_pM$, $p\in M$. For a given orthonormal basis $e_1,\ldots,e_n$ of the tangent space $T_pM$, the scalar curvature $\tau$ at $p$ is defined to be $$\tau(p)=\sum_{i<j} K(e_i\wedge e_j).\notag$$

Let $L$ be a subspace of $T_pM$  of dimension $r\geq 2$  and let $\{e_1,\ldots,e_r\}$ be an orthonormal basis of $L$. We define the scalar curvature $\tau(L)$ of $L$ by  $$\tau(L)=\sum_{\alpha<\beta} K(e_\alpha\wedge e_\beta),\quad 1\leq
\alpha,\beta\leq r.$$

Given an integer $k\geq 1$, we denote by ${\mathcal S}(n,k)$ the finite set  consisting of unordered $k$-tuples $(n_1,\ldots,n_k)$ of integers $\geq 2$ satisfying  $n_1< n$ and $n_1+\cdots+n_k\leq n$. We put ${\mathcal S}(n)=\cup_{k\geq 1}{\mathcal S}(n,k)$. 

 For each $k$-tuple $(n_1,\ldots,n_k)\in {\mathcal S}(n)$, the author introduced the $\delta$-invariant $\delta{(n_1,\ldots,n_k)}$ as (cf. \cite{c98,c00,book})
\begin{align}\notag \delta(n_1,\ldots,n_k)(p)=\tau(p)-\inf\{\tau(L_1)+\cdots+\tau(L_k)\}, \end{align}
where $L_1,\ldots,L_k$ run over all $k$ mutually orthogonal subspaces of $T_pM$ such that  $\dim L_j=n_j,\, j=1,\ldots,k$. 

The $\delta$-curvatures are very different in nature from the ``classical'' scalar and Ricci curvatures; simply due to the fact that both scalar and Ricci curvatures are the ``total sum'' of sectional curvatures on a Riemannian manifold. In contrast, the $\delta$-curvature invariants  are obtained  from the scalar curvature by throwing away a certain amount of sectional curvatures. 
(For the history and motivation on $\delta$-invariants, see author's most recent survey article \cite{c13}.)

\subsection{Fundamental inequalities}
The author proved the following fundamental inequalities in \cite{c98,c00}.

\begin{Theorem A} Let $M^n$ be an $n$-dimensional submanifold in a real space form $R^m(\epsilon)$ of constant  curvature $\epsilon$. Then, for each $k$-tuple  $(n_1,\ldots,n_k)\in\mathcal S(n)$, we have
\begin{align}\label{2.8} \delta{(n_1,\ldots,n_k)} \leq  {{n^2(n+k-1-\sum n_j)}\over{2(n+k-\sum n_j)}}H^2 +{1\over2} \Big({{n(n-1)}}-\sum_{j=1}^k {{n_j(n_j-1)}}\Big)\epsilon.\end{align}

The equality case of inequality \eqref{2.8}  holds at a point $p\in M$ if and only if, there exists an orthonormal basis  $\{e_1,\ldots,e_m\}$ at $p$, such that  the shape operators of $M$ in $R^m(\epsilon)$ at $p$  with respect to $\{e_1,\ldots,e_m\}$  take the form:
\begin{align} \font\b=cmr10 scaled \magstep2
\def\bigzerol{\smash{\hbox{ 0}}}
\def\bigzerou{\smash{\lower.0ex\hbox{\b 0}}} A_r=\left( \begin{matrix} A^r_{1} & \hdots & 0
\\ \vdots  & \ddots& \vdots &\bigzerou \\ 0 &\hdots &A^r_k&
\\ \\&\bigzerou & &\mu_rI \end{matrix} \right),\quad  r=n+1,\ldots,m,
\label{2.9}\end{align}
where $I$ is an identity matrix and $A^r_j$ is a symmetric $n_j\times n_j$  submatrix satisfying
$$\hbox{\rm trace}\,(A^r_1)=\cdots=\hbox{\rm
trace}\,(A^r_k)=\mu_r.$$\end{Theorem A}

In particular, for hypersurfaces in a Euclidean 4-space, Theorem A implies the following.

\begin{theorem} \label{T:2.1} Let $M$ be an $3$-dimensional  submanifold of a Riemannian $4$-manifold $R^4(\epsilon)$ of constant sectional curvature $\epsilon$. Then 
\begin{align}\label{2.10}\delta_M \leq\frac{9}{4}  H^2 +2\epsilon.\end{align}

Equality case of \e{2.10} hold if and only if, with respect to suitable orthonormal frame $\{e_1,e_2,e_3,e_4\}$, the shape operator $A= A_{e_4}$ of $M$ in $R^4(\epsilon)$ take the following form:
\begin{align}\label{2.11}A=\left( \begin{matrix} \lambda & 0 & 0  \\ 0 &\mu & 0  \\ 0 & 0 &\lambda+\mu \end{matrix}\right)\end{align}
for some functions $\lambda$ and $\mu$.
\end{theorem}

 A submanifold of a Euclidean space is called {\it $\delta(n_{1},\ldots,n_{k})$-ideal} if it satisfies the equality case of \e{2.8} identically. Roughly speaking, an ideal immersion is a very nice immersion which produces the least possible amount of tension from the ambient space.  Such submanifolds have many interesting properties and have been studied by many geometers during the last two decades (see \cite{c08,book} for details).

Since the invariant $\delta_M$ defined in \e{1.1} is the only non-trivial $\delta$-invariant for Riemannian 3-manifolds,  an isometric immersion of a 3-manifold $M$ is ideal if and only if it is $\delta(2)$-ideal, i.e., it satisfied the equality case of \e{2.10} identically.

\section{ Brief reviews of Jacobi's elliptic functions}

We review briefly some known facts on Jacobi's elliptic functions for later use (for details,
see, for instance, \cite{B}).

Put
\begin{align}\label{3.1} &u =\int_0^x \frac{dt}{\sqrt{(1-t^2)(1-k^2t^2)}},
\\&\label{3.2} K=\int_0^1\frac{dt}{\sqrt{(1-t^2)(1-k^2t^2)}},\end{align} where
we first suppose that $x$ and $k$ satisfy $0<k<1$ and $-1\leq x\leq 1$. 

Equation \e{3.1} defines $u$ as an odd function of $x$ which is
positive, increasing from $0$ to $K$ as $x$ increases from $0$ to
1. Inversely, the same equation defines $x$ as an odd function of
$u$ which increases from 0 to 1 as $u$ increase from 0 to $K$;
this function  is known as a Jacobi's elliptic function, denoted
by sn$(u,k)$  (or simply by sn$(u)$), so that we can put
\begin{align}\label{3.3}u=\hbox{sn}^{-1}(x),\quad x=\hbox{sn}(u).\end{align}
 The other two main Jacobi's
functions sn$(u,k)$ and dn$(u,k)$ (or simply denoted
respectively by sn$(u)$ and dn$(u)$) are defined by
\begin{align}\label{3.4}\hbox{cn}(u)=\sqrt{1-\hbox{sn}^2(u)},\quad 
\hbox{dn}(u)=\sqrt{1-k^2\hbox{sn}^2(u)},\end{align} the square
roots are positive so long as $u$ is confined to $-K<u<K$, so that cn$(u)$ and dn$(u)$ are even functions of $u$.
Let $k'=\sqrt{1-k^2}$ be the complementary modulus. Then
$\dn(u)\geq k'>0$.
The Jacobi's elliptic functions depend on the variable $u$ as well
as on the parameter $k$, which is called the {\it modulus}. 

It is well-known that the Jacobi's elliptic functions satisfy
the following identities:
\begin{equation} \begin{aligned}\label{3.5}&\sn^2(u)+\cn^2(u)=1,\quad \dn^2(u)+k^2\sn^2(u)=1,\\&k^2\cn^2(u)+{k'}^2=\dn^2(u),\quad\cn^2(u)+{k'}^2\sn^2(u)=\dn^2(u).\end{aligned}\end{equation}

It is also known that  the Jacobi's
elliptic functions satisfy
\begin{equation} \begin{aligned}\label{3.6} 
&\frac{d}{du}\sn(u)=\cn(u)\dn(u),\;\;\frac{d}{du}\cn(u)=-\sn(u)\dn(u),\\ 
&\frac{d}{du}\dn(u)=-k^2\sn(u)\cn(u).
\end{aligned}\end{equation}

Using $\cn(u),\, dn(u)$ and $\sn(u)$, one may define minor Jacobi elliptic functions as follows:
\begin{align} \cd(u)=\frac{\cn(u)}{\dn(u)},\; \sd(u)=\frac{\sn(u)}{\dn(u)},\; \ns(u)=\frac{1}{\sn(u)},\cdots,etc.\end{align}

\section{Ideal hypersurfaces with two distinct principal curvatures}

In this section, we completely classify all ideal hypersurfaces with two distinct principal curvatures in $\mathbb E^4$.

\begin{theorem}\label{T:4.1} Let $M$ be an ideal hypersurface of the Euclidean 4-space $\mathbb E^4$. Then $M$ has two distinct principal curvatures at each point if and only if $M$ is congruent to one of the following hypersurfaces:
\begin{enumerate}
\item[{\rm (a)}] A spherical cylinder given by
\begin{align} \big(t, a \sin u,a\cos u\sin v, a\cos u\cos v\big)\end{align}
for some positive number $a$;
\item[{\rm (b)}] A cone given by
\begin{align}   \Big(\sqrt{1-a^2} t, at \sin u,at\cos u\sin v, at\cos u\cos v\Big)\end{align}
for some real number $a$ satisfying $0\leq a\leq 1$;
\
\item[{\rm (c)}] A hypersurface given by
\begin{equation}\begin{aligned} &\Bigg(\frac{1}{a} \sd \big(at,\tfrac{1}{\sqrt{2}}\big) \sin u,\frac{1}{a}\sd \big(at,\tfrac{1}{\sqrt{2}}\big)\cos u\sin v, \frac{1}{a}\sd \big(at,\tfrac{1}{\sqrt{2}}\big)\cos u\cos v,
\\&\hskip1in \frac{1}{2}
  \int_0^t \sd^2\big(at,\tfrac{1}{\sqrt{2}}\big)dt\Bigg)\end{aligned}\end{equation}
for some positive real number $a$.\end{enumerate}
\end{theorem}
\begin{proof} Assume that $M$ is an ideal hypersurface of the Euclidean 4-space. Then Theorem \ref{T:2.1} implies that there exists an orthonormal frame $\{e_1,e_2,e_3,e_4\}$ such that the shape operator of $M$ with respect to this frame takes the following simple form:
\begin{align}\label{4.4}A=\left( \begin{matrix} \lambda & 0 & 0  \\ 0 &\mu & 0  \\ 0 & 0 &\lambda+\mu \end{matrix}\right)\end{align}
for some functions $\lambda$ and $\mu$.

Let $\omega_i^j$ be the connection forms defined by
\begin{align}\label{4.5} \nabla_X e_i=\sum_{j=1}^3 \omega_i^j(X)e_j,\;\; i=1,2,3.\end{align}
Then we have $\omega_i^j=-\omega_j^i$ for $i,j=1,2,3$. In particular, we have $\omega_i^i=0$.

Now, let us assume that $M$ has two distinct principal curvatures at each point. Then one of the following three cases must occurs: 
(i) $\lambda=\mu$, 
(ii) $\lambda=0$, or 
(iii) $\mu=0$.

\vskip.05in
{\it Case} (i): $\lambda=\mu$: In this case, the second fundamental form satisfies
\begin{equation} \begin{aligned} \label{4.6}&h(e_1,e_1)=h(e_2,e_2)=\lambda e_4,\;\\&h(e_3,e_3)=2\lambda e_4,\\& h(e_i,e_j)=0,\;\; otherwise.\end{aligned}\end{equation}

By straight-forward computation, we find  the following equations from \e{4.5}, \e{4.6} and the equation of Codazzi.
\begin{align} &\label{4.7} e_1\lambda=e_2\lambda =0, \;\; e_3\lambda=\lambda \omega_3^1(e_1)=\lambda\omega_3^2(e_2),
\\& \label{4.8}\omega_3^1(e_3)=\omega_3^2(e_3)=0,
\\&\label{4.9} \omega_2^3(e_1)=\omega_1^3(e_2)=0.
\end{align}

Let $\mathcal D$ denote the distribution spanned by $e_1$ and $e_2$. It follows from \e{4.9} that the distribution $\mathcal D$ is an integrable distribution. Moreover, we know from 
\e{4.7} and \e{4.9} that every leave of $\mathcal D$ is a totally umbilical surface in $M$ with constant mean curvature. Thus $\mathcal D$ is a spherical distribution. Furthermore, it follows from \e{4.8} that the integral curves of $e_3$ are geodesic in $N$. Therefore, the distribution spanned by $e_3$ is a totally geodesic distribution. 

Let $N$ be a leave of $\mathcal D$. Since $N$ is totally umbilical in $M$,  \e{4.6} implies that $N$ is also a totally umbilical surface in $\mathbb E^4$. Therefore $N$ is an open portion of 2-sphere. Hence we may apply a result of Hiepko to conclude that $M$ is locally a warped product ${\bf R}\times_f S^2(1)$ of a real line and the unit 2-sphere $S^2(1)$ with a warping function $f$ on ${\bf R}$ (cf. \cite{H} or \cite[page 90]{book}). Consequently, we may assume that the metric tensor of $M$ is given by
\begin{align}\label{4.10} g=dt^2 +f^2(t)(du^2+(\cos^2 u) dv^2)\end{align} 

Obviously, $e_3$ is tangent to the first factor and $e_1,e_2$ are tangent to the second factor of the warped product.  Thus we may assume that
\begin{align} \label{4.11} e_1=\frac{1}{f}\frac{\partial}{\partial u},\;\;  e_2=\frac{\sec u}{f}\frac{\partial}{\partial v},\;\; e_3=\frac{\partial}{\partial t}.\end{align}

By combining \e{4.7} and \e{4.11} we see that $\lambda=\lambda(t)$. Thus we find from \e{4.7} that
\begin{align}\label{4.12} \omega_3^1(e_1)=\omega_3^2(e_2)=(\ln \lambda)' .\end{align}
From \e{4.8}, \e{4.9} and \e{4.12} we obtain
\begin{align} \label{4.13} \nabla_{e_1}e_3=\frac{\lambda'}{\lambda} e_1,\;\; \nabla_{e_2}e_3=\frac{\lambda'}{\lambda} e_2,\;\; \nabla_{e_3}e_3=0,\end{align}
which implies that the curvature tensor $R$ of $M$ satisfies
\begin{align}\label{4.14} \<R(e_1,e_3)e_3,e_1\>=-(\ln\lambda)'' -(\ln \lambda')^2.\end{align}

On the other hand, we find from \e{4.6} and the equation of Gauss that
\begin{align}\label{4.15} \<R(e_1,e_3)e_3,e_1\>=2\lambda^2.\end{align}
So, after combining \e{4.14} and \e{4.15}, we obtain the following differential equation:
\begin{align}\label{4.16} \lambda''+2\lambda^3=0.\end{align}
By solving this second order non-linear differential equation, we get $$\lambda(t)=\frac{a}{2}\sd\Big(at+b,\frac{1}{\sqrt{2}}\Big)$$ for some positive number $a$ and a real number $b$. Therefore, after applying a suitable translation in $t$, we have
\begin{align} \label{4.17} \lambda(t)=\frac{a}{2}\sd\Big(at,\text{\small$\frac{1}{\sqrt{2}}$}\Big).\end{align}
Now, by using \e{4.6}, \e{4.11} and \e{4.17} we derive that
\begin{equation} \begin{aligned} \label{4.18}&h\(\frac{\partial }{\partial u},\frac{\partial }{\partial u}\)=\frac{a}{2}f^2 \sd\Big(at,\text{\small$\frac{1}{\sqrt{2}}$}\Big) e_4,
\\&h\(\frac{\partial }{\partial v},\frac{\partial }{\partial v}\)=\frac{a}{2}f^2\cos^2 u \sd\Big(at,\text{\small$\frac{1}{\sqrt{2}}$}\Big) e_4,
\\&h\(\frac{\partial }{\partial t},\frac{\partial }{\partial t}\)=a \sd\Big(at,\text{\small$\frac{1}{\sqrt{2}}$}\Big) e_4,
\\& h\(\frac{\partial }{\partial t},\frac{\partial }{\partial u}\)=h\(\frac{\partial }{\partial t},\frac{\partial }{\partial v}\)=g\(\frac{\partial }{\partial u},\frac{\partial }{\partial v}\)=0.\end{aligned}\end{equation}
Moreover, after a straight-forward long computation, we know from \e{4.10} that the Levi-Civita connection of $M$ satisfies
\begin{equation}\begin{aligned}\label{4.19}
 &\nabla_{\frac{\partial}{\partial t}}\frac{\partial}{\partial t}= 0,
\;\; \nabla_{\frac{\partial}{\partial t}}\frac{\partial}{\partial u}=\frac{ f'} {f}\frac{\partial }{\partial u},\;\; \nabla_{\frac{\partial}{\partial t}}\frac{\partial}{\partial v}=\frac{ f'} {f}\frac{\partial }{\partial v},
\\& \nabla_{\frac{\partial}{\partial u}}\frac{\partial}{\partial u}=-f f'\frac{\partial}{\partial t},
\;\;  \nabla_{\frac{\partial}{\partial u}}\frac{\partial}{\partial v}=- \tan u \frac{\partial }{\partial v} ,
\\& \nabla_{\frac{\partial}{\partial v}}\frac{\partial}{\partial v}=-f f' \cos^2 u \frac{\partial}{\partial t}+\sin u\cos u\frac{\partial}{\partial u}. \end{aligned}\end{equation} 

Now, by applying \e{4.18}, \e{4.19} and the following equation
$$(\bar \nabla_{\frac{\partial}{\partial t}}h)\(\frac{\partial}{\partial u},\frac{\partial}{\partial u}\)= (\bar \nabla_{\frac{\partial}{\partial u}}h)\(\frac{\partial}{\partial t},\frac{\partial}{\partial u}\)$$
of Codazzi, we find 
\begin{align}\label{4.20} \frac{f'}{f}=a \cd\Big(at,\text{\small$\frac{1}{\sqrt{2}}$}\Big)  \ns\Big(at,\text{\small$\frac{1}{\sqrt{2}}$}\Big). \end{align}
After solving this differential equation, we get
\begin{align}\label{4.21} f(t)=c\sd \Big(at,\text{\small$\frac{1}{\sqrt{2}}$}\Big) \end{align}
for some nonzero constant $c$. 

By applying \e{4.6}, \e{4.17}, \e{4.19}, we see that the sectional curvature $K(\frac{\partial}{\partial u}\wedge \frac{\partial}{\partial v}) $ of the plane section spanned by $\frac{\partial}{\partial u}$ and $\frac{\partial}{\partial v}$ satisfies
\begin{align}\label{4.22} \lambda^2=K(\tfrac{\partial}{\partial t}\wedge \tfrac{\partial}{\partial u})=\frac{1-f'{}^2}{f^2}.\end{align}
Now, by substituting \e{4.17} and \e{4.21} into \e{4.22} we find $c^2=a^{-2}$. Thus, without of generality, we may put $c=a^{-1}$. Consequently, we have
\begin{align}\label{4.23}  f(t)=\frac{1}{a}\sd \Big(at,\text{\small$\frac{1}{\sqrt{2}}$}\Big). \end{align}
By combining this with \e{4.10} we obtain
\begin{align}\label{4.24} g=dt^2 +\frac{\sd^2 \big(at,\text{\small$\frac{1}{\sqrt{2}}$}\big)}{a^2}(du^2+\cos^2 u\, dv^2), \end{align} 
which implies that
\begin{equation}\begin{aligned}\label{4.25}
 &\nabla_{\frac{\partial}{\partial t}}\frac{\partial}{\partial t}= 0,
\;\; \\& \nabla_{\frac{\partial}{\partial t}}\frac{\partial}{\partial u}=a \cd\Big(at,\text{\small$\frac{1}{\sqrt{2}}$}\Big)  \ns\Big(at,\text{\small$\frac{1}{\sqrt{2}}$}\Big)\frac{\partial }{\partial u},\;\; \\& \nabla_{\frac{\partial}{\partial t}}\frac{\partial}{\partial v}=a \cd\Big(at,\text{\small$\frac{1}{\sqrt{2}}$}\Big)  \ns\Big(at,\text{\small$\frac{1}{\sqrt{2}}$}\Big)\frac{\partial }{\partial v},
\\& \nabla_{\frac{\partial}{\partial u}}\frac{\partial}{\partial u}=- \frac{1}{a} \cd\Big(at,\text{\small$\frac{1}{\sqrt{2}}$}\Big)\sd\Big(at,\text{\small$\frac{1}{\sqrt{2}}$}\Big)\nd\Big(at,\text{\small$\frac{1}{\sqrt{2}}$}\Big)\frac{\partial}{\partial t},
\;\;  \\& \nabla_{\frac{\partial}{\partial u}}\frac{\partial}{\partial v}=- \tan u \frac{\partial }{\partial v} ,
\\& \nabla_{\frac{\partial}{\partial v}}\frac{\partial}{\partial v}= - \frac{1}{a} \cd\Big(at,\text{\small$\frac{1}{\sqrt{2}}$}\Big)\sd\Big(at,\text{\small$\frac{1}{\sqrt{2}}$}\Big)\nd\Big(at,\text{\small$\frac{1}{\sqrt{2}}$}\Big) \cos^2 u \frac{\partial}{\partial t}\\&\hskip.8in +\sin u\cos u\frac{\partial}{\partial u}.  \end{aligned}\end{equation} 
Moreover, it follows from \e{4.6}, \e{4.11} and \e{4.17} that
\begin{equation} \begin{aligned} \label{4.26}&h\(\frac{\partial }{\partial u},\frac{\partial }{\partial u}\)=\frac{1}{2a} \sd^3\Big(at,\text{\small$\frac{1}{\sqrt{2}}$}\Big) e_4,
\\&h\(\frac{\partial }{\partial v},\frac{\partial }{\partial v}\)=\frac{1}{2a}\cos^2 u \sd^3\Big(at,\text{\small$\frac{1}{\sqrt{2}}$}\Big) e_4,
\\&h\(\frac{\partial }{\partial t},\frac{\partial }{\partial t}\)=a \sd\Big(at,\text{\small$\frac{1}{\sqrt{2}}$}\Big) e_4,
\\& h\(\frac{\partial }{\partial t},\frac{\partial }{\partial u}\)=h\(\frac{\partial }{\partial t},\frac{\partial }{\partial v}\)=g\(\frac{\partial }{\partial u},\frac{\partial }{\partial v}\)=0.\end{aligned}\end{equation}
Therefore, by using the formula of Gauss, \e{4.25} and \e{4.26}, we may conclude that the immersion $L: M\to \mathbb E^4$ of the ideal hypersurface satisfies
\begin{align}\label{4.27}
 &\frac{\partial^2 {L}}{\partial t^2}= a \sd\Big(at,\text{\small$\frac{1}{\sqrt{2}}$}\Big) e_4,
 \\&\label{4.28}\frac{\partial^2 L}{\partial t \partial u}=a \cd \Big(at,\text{\small$\frac{1}{\sqrt{2}}$}\Big)  \ns\Big(at,\text{\small$\frac{1}{\sqrt{2}}$}\Big)\frac{\partial L}{\partial u},
\\&\label{4.29} \frac{\partial^2 L}{\partial t \partial v}=a \cd\Big(at,\text{\small$\frac{1}{\sqrt{2}}$}\Big)  \ns\Big(at,\text{\small$\frac{1}{\sqrt{2}}$}\Big) \frac{\partial L}{\partial v},
\\&\label{4.30}  \frac{\partial^2 L}{\partial u \partial u}=- \frac{1}{a} \cd\Big(at,\!\text{\small$\frac{1}{\sqrt{2}}$}\Big)\sd\Big(at,\!\text{\small$\frac{1}{\sqrt{2}}$}\Big)\nd\Big(at,\!\text{\small$\frac{1}{\sqrt{2}}$}\Big)\frac{\partial L}{\partial t}\\&\notag\hskip.8in +\frac{1}{2a} \sd^3\Big(at,\!\text{\small$\frac{1}{\sqrt{2}}$}\Big) e_4,
\\& \label{4.31}  \frac{\partial^2 L}{\partial u \partial v} =- \tan u \frac{\partial L}{\partial v} ,
\\& \label{4.32}  \frac{\partial^2 L}{\partial v \partial v}= - \frac{1}{a} \cd\Big(at,\text{\small$\frac{1}{\sqrt{2}}$}\Big)\sd\Big(at,\text{\small$\frac{1}{\sqrt{2}}$}\Big)\nd\Big(at,\text{\small$\frac{1}{\sqrt{2}}$}\Big) \cos^2 u \frac{\partial L}{\partial t}
\\&\notag\hskip.6in +\sin u\cos u\frac{\partial L}{\partial u}+\frac{1}{2a}\cos^2 u \sd^3\Big(at,\text{\small$\frac{1}{\sqrt{2}}$}\Big) e_4.  \end{align} 

After solving \e{4.31} we get
\begin{align}\label{4.33} L(t,u,v)=A(t,v)\cos u+B(t,u)\end{align}
for some vector-valued functions $A(t,v)$ and $B(t,u)$. Now, by substituting \e{4.33} into \e{4.29} we find 
\begin{equation}\begin{aligned}\label{4.34}
 &\frac{\partial^2 A}{\partial t\partial v}= a \cd\Big(at,\text{\small$\frac{1}{\sqrt{2}}$}\Big)\ns\Big(at,\text{\small$\frac{1}{\sqrt{2}}$}\Big) \frac{\partial A}{\partial v},
 \end{aligned}\end{equation}
which implies  \begin{equation}\begin{aligned}\label{4.35}
 &A(t,v)= P(t) +Q(v)\sd\Big(at,\text{\small$\frac{1}{\sqrt{2}}$}\Big)
 \end{aligned}\end{equation}
for some vector functions $P,Q$. Combining \e{4.35} with \e{4.33} gives
\begin{align}\label{4.36} L(t,u,v)=(\cos u)\Big(P(t) +Q(v)\sd\Big(at,\text{\small$\frac{1}{\sqrt{2}}$}\Big)\Big)+B(t,u).\end{align}

Also, after substituting \e{4.36} into \e{4.28} we obtain
\begin{align}\label{4.37} &\sn\Big(at,\text{\small$\frac{1}{\sqrt{2}}$}\Big)P'(t)=a\cd\Big(at,\text{\small$\frac{1}{\sqrt{2}}$}\Big)P(t),
\\& \label{4.38} \sn\Big(at,\text{\small$\frac{1}{\sqrt{2}}$}\Big)\frac{\partial^2 B}{\partial t\partial u}=a\cd\Big(at,\text{\small$\frac{1}{\sqrt{2}}$}\Big)\frac{\partial B}{\partial u}.
\end{align}
By solving the differential equations \e{4.37} and \e{4.38} we find
\begin{align}\label{4.39} & P(t) =c_0\sd\Big(at,\text{\small$\frac{1}{\sqrt{2}}$}\Big),
\\& \label{4.40} B(t,u)=R(u)\sd\Big(at,\text{\small$\frac{1}{\sqrt{2}}$}\Big)+S(t),
 \end{align} for some vector $c_0$ and vector functions $R(u),S(t)$.
After combining \e{4.39} and \e{4.40} with \e{4.36} we get
\begin{align}\label{4.41} L(t,u,v)=S(t)+(R(u) + T(v)\cos u)\sd\Big(at,\text{\small$\frac{1}{\sqrt{2}}$}\Big),\end{align}
where $T(v)=c_0+Q(v)$. Now, by substituting \e{4.41} into \e{4.27} we get
\begin{equation}\label{4.42} e_4=\frac{1}{a} S''(t)\ds\Big(at,\text{\small$\frac{1}{\sqrt{2}}$}\Big)-a(R(u)+T(v)\cos u)\sd^2\Big(at,\text{\small$\frac{1}{\sqrt{2}}$}\Big).
\end{equation}
So, after substituting \e{4.41} and \e{4.42} into \e{4.30}, we obtain
\begin{equation}\begin{aligned}\label{4.43}
 & \hskip.5in 2a^2(R''(u)+R(u))\dn^4\Big(at,\text{\small$\frac{1}{\sqrt{2}}$}\Big)
 \\&=\dn^2\Big(at,\text{\small$\frac{1}{\sqrt{2}}$}\Big)\Big(S''(t)\dn\Big(at,\text{\small$\frac{1}{\sqrt{2}}$}\Big)\sn\Big(at,\text{\small$\frac{1}{\sqrt{2}}$}\Big)-2a S'(t)\cn\Big(at,\text{\small$\frac{1}{\sqrt{2}}$}\Big)\Big). \end{aligned}\end{equation}
It follows from \e{4.43} that 
\begin{align}\label{4.44} &  R''(u)+R(u)=d_1  \end{align}
for some vector $d_1$.
 By solving \e{4.44} we get $$R(u)=d_1+d_2 \cos u+c_1 \sin u$$ for some vectors $d_2,c_1$. Combining this with \e{4.41} yields
\begin{align}\label{4.45} L(t,u,v)=G(t)+(c_1 \sin u+ H(v)\cos u)\sd\Big(at,\text{\small$\frac{1}{\sqrt{2}}$}\Big)\end{align} with $G(t)=S(t)+d_1 \sd\Big(at,\text{\small$\frac{1}{\sqrt{2}}$}\Big)$ and $H(v)=d_2+T(v)$.

Substituting \e{4.45} into \e{4.27} gives
\begin{equation}\label{4.46} e_4=\frac{1}{a} G''(t)\ds\Big(at,\text{\small$\frac{1}{\sqrt{2}}$}\Big)-a(c_1 \sin u+H(v)\cos u)\sd^2\Big(at,\text{\small$\frac{1}{\sqrt{2}}$}\Big).\end{equation}
Finally, by substituting \e{4.45} and \e{4.46} into \e{4.30} and \e{4.32}, we obtain after long computation that
\begin{align}\notag L=(c_1 \sin u+ (c_2 \cos v+c_3\sin v))\cos u)\sd\Big(at,\text{\small$\frac{1}{\sqrt{2}}$}\Big)+c_4\int_0^t \sd^2\Big(as,\text{\small$\frac{1}{\sqrt{2}}$}\Big)ds\end{align}
for some vectors $c_1,\ldots,c_4\in \mathbb E^4$.
Consequently,  by choosing a suitable coordinate system of $\mathbb E^4$, we obtain case (c) of the theorem.

\vskip.05in
{\it Case} (ii): $\lambda=0$.  In this case, the second fundamental form satisfies
\begin{equation} \begin{aligned} \label{4.47}&h\(e_2,e_2\)=\mu e_4,
\;\; h(e_3,e_3)=\mu e_4,\\& h(e_i,e_j)=0,\;\; otherwise.\end{aligned}\end{equation}
From \e{4.5}, \e{4.47} and Codazzi's equation we obtain
\begin{align} &\label{4.48} e_2\mu=e_3\mu =0, \;\; e_1\mu=\mu \omega^1_2(e_2)=\mu\omega^1_3(e_3),
\\& \label{4.49}\omega_2^1(e_3)=\omega^1_3(e_2)=0,
\\&\label{4.50} \omega_1^2(e_1)=\omega_1^3(e_1)=0.
\end{align}

Let $\mathcal H$ be the distribution spanned by $e_2$ and $e_3$. It follows from \e{4.48}-\e{4.50} that  $\mathcal H$ is an integrable distribution whose leaves are totally umbilical in $M$ with constant mean curvature. Thus, $\mathcal H$ is a spherical distribution. Also, it follows from \e{4.50} that the integral curves of $e_1$ are geodesic in $N$.  Therefore, Hiepko's theorem in \cite{H} implies that $M$ is locally a warped product ${\bf R}\times_f S^2(1)$ of a real line and a unit 2-sphere $S^2(1)$. Consequently, we may assume that the metric tensor of $M$ is given by
\begin{align}\label{4.51} g=dt^2 +f^2(t)(du^2+\cos^2 u\, dv^2).\end{align} 
Obviously, $e_1$ is tangent to the first factor and $e_2,e_3$ are tangent to the second factor of the warped product.  Thus we have
\begin{align} \label{4.52} e_1=\frac{\partial}{\partial t},\;\; e_2=\frac{1}{f}\frac{\partial}{\partial u},\;\;  e_2=\frac{\sec u}{f}\frac{\partial}{\partial v}.\end{align}
From \e{4.51} we conclude that the Levi-Civita connection $\nabla$ of $M$ satisfies \e{4.18}. Moreover, \e{4.48} shows that $\mu=\mu(t)$.

It follows from \e{4.18} that the sectional curvature $K(\pi)$ of the plane section $\pi$ spanned by $\frac{\partial}{\partial t}, \frac{\partial}{\partial u}$ is equal to $-f''/f$. On the other hand, it follows from \e{4.47} and Gauss' equation that $K(\pi)=0$. Therefore we get $f''=0$, which implies that $f=at+b$ for some real numbers $a,b$, not both zero.  

If $a\ne 0$, then after applying a suitable translation in $t$ we have $f=a t$. Consequently,  either ($\alpha$) $f=b$ with $b\ne 0$ or ($\beta$) $f=at$ with $a\ne 0$.

\vskip.05in
{\it Case} (ii.$\alpha$): $f=b, \, b\ne 0$. In this case, \e{4.51} becomes
\begin{align}\label{4.53} g=dt^2 +b^2(du^2+\cos^2 u\, dv^2).\end{align} 
Thus $M$ is an open portion of the Riemannian product of a line and a 2-sphere $S^2(b)$ with radius $b$.
Hence, in view of \e{4.47}, we conclude that the immersion $L:M\subset {\bf R}\times S^2(\frac{1}{b})\to \mathbb E^4$ is the product immersion of a line and an ordinary 2-sphere $S^2(\frac{1}{b})$ in $\mathbb E^3$. Clearly, in this case the second fundamental form of $M$ in $\mathbb E^4$ depends only the metric tensor of $M$.

\vskip.05in
{\it Case} (ii.$\beta$): $f=at$. In this case, \e{4.51} becomes
\begin{align}\label{4.54} g=dt^2 +a^2 t^2(du^2+\cos^2 u\, dv^2).\end{align}
Without loss of generality, we may assume that $a$ is positive. Thus the Levi-Civita connection of $g$ satisfies
\begin{equation}\begin{aligned}\label{4.55}
 &\nabla_{\frac{\partial}{\partial t}}\frac{\partial}{\partial t}= 0,
\;\; \nabla_{\frac{\partial}{\partial t}}\frac{\partial}{\partial u}=\frac{ 1} {t}\frac{\partial }{\partial u},\;\; \nabla_{\frac{\partial}{\partial t}}\frac{\partial}{\partial v}=\frac{1} {t}\frac{\partial }{\partial v},
\\& \nabla_{\frac{\partial}{\partial u}}\frac{\partial}{\partial u}=-a^2t\frac{\partial}{\partial t},
\;\;  \nabla_{\frac{\partial}{\partial u}}\frac{\partial}{\partial v}=- \tan u \frac{\partial }{\partial v} ,
\\& \nabla_{\frac{\partial}{\partial v}}\frac{\partial}{\partial v}=-a^2t \cos^2 u \frac{\partial}{\partial t}+\sin u\cos u\frac{\partial}{\partial u}. \end{aligned}\end{equation} 
It follows from \e{4.55} that the sectional curvature $K(\hat\pi)$ of the plane section $\hat\pi$ spanned by $
\frac{\partial }{\partial u},\frac{\partial }{\partial v}$ is equal to $(1-a^2)/(a^2t^2)$. 

On the other hand, the equation of Gauss gives $K(\hat\pi)=\mu^2$. Therefore, we may put
\begin{align}\label{4.56} \mu=\frac{\sqrt{1-a^2}}{at}\end{align}
for some positive number $0<a<1$.
Consequently,  \e{4.47} becomes
\begin{equation} \begin{aligned} \label{4.57}&h\(\frac{\partial }{\partial t},\frac{\partial }{\partial t}\)=0,
\; h\(\frac{\partial }{\partial u},\frac{\partial }{\partial u}\)=a\sqrt{1-a^2}t e_4,
\;\\&h\(\frac{\partial }{\partial v},\frac{\partial }{\partial v}\)=a\sqrt{1-a^2}t\cos^2 u e_4,
\\& h\(\frac{\partial }{\partial t},\frac{\partial }{\partial u}\)=h\(\frac{\partial }{\partial t},\frac{\partial }{\partial v}\)=h\(\frac{\partial }{\partial u},\frac{\partial }{\partial v}\)=0.\end{aligned}\end{equation}
 
 Gauss' formula, \e{4.55} and \e{4.57} imply that the immersion $L: M\to \mathbb E^4$ of the ideal hypersurface satisfies
\begin{align}\label{4.58}
 &\frac{\partial^2 {L}}{\partial t^2}= 0,
\;\;\frac{\partial^2 L}{\partial t \partial u}=\frac{1}{t}\frac{\partial L}{\partial u},
\;\;  \frac{\partial^2 L}{\partial t \partial v}=\frac{1}{t} \frac{\partial L}{\partial v},
\\&\label{4.59}  \frac{\partial^2 L}{\partial u \partial u}=-a^2 t\frac{\partial L}{\partial t} +a\sqrt{1-a^2}t e_4,
\\& \label{4.60}  \frac{\partial^2 L}{\partial u \partial v} =- \tan u \frac{\partial L}{\partial v} ,
\\& \label{4.61}  \frac{\partial^2 L}{\partial v \partial v}=-a^2 t\cos^2 u \frac{\partial L}{\partial t} +\sin u\cos u\frac{\partial L}{\partial u}+a\sqrt{1-a^2}t\cos^2 u e_4.  \end{align} 
Moreover,  \e{4.54}, \e{4.56}, and Weingarten's formula imply
\begin{equation}\begin{aligned}\label{4.62} &\frac{\partial e_4}{\partial t}=0,
\;\;\frac{\partial e_4 }{\partial u}=-\frac{\sqrt{1-a^2}}{at} \frac{\partial L}{\partial u},
\;\; \frac{\partial e_4 }{\partial v}=-\frac{\sqrt{1-a^2}}{at} \frac{\partial L}{\partial v}. \end{aligned}\end{equation}

Solving \e{4.58} gives
\begin{align}\label{4.63} L(t,u,v)=t A(u,v)\end{align}
for some vector function $A(u,v)$. So, after substituting \e{4.63} into \e{4.60} we find
$\frac{\partial^2 A}{\partial u \partial v} =- \tan u \frac{\partial A}{\partial v}$, which implies that
$$A(u,v)=P(u)+Q(v)\cos u$$ for some vector functions $P(u), Q(v)$. Combining this with \e{4.63}  gives
\begin{align}\label{4.64} L(t,u,v)=t (P(u)+Q(v)\cos u).\end{align}
Now, by substituting \e{4.64} into \e{4.60} and \e{4.61}, we find
\begin{align}\label{4.65} &(\cos u) P''(u)+(\sin u) P'(u)=-c_0,
\\&\label{4.66} Q''(v)+Q(v)=-c_0,\end{align}
for some vector  $c_0\in \mathbb E^4$. After solving \e{4.65} and \e{4.66} we get
\begin{align}\label{4.67} &P(u)=c_0 \cos u+c_2\sin u+c_1,
\\&\label{4.68} Q(v)=c_3 \cos v +c_4 \sin v-c_0,\end{align}
for some vectors $c_1,c_2,c_3,c_4$. Now, by combining \e{4.64}, \e{4.67} and \e{4.68}, we obtain
\begin{align}\label{4.69} L(t,u,v)=t (c_1+c_2 \sin u+(c_3 \cos v+c_4\sin v)\cos u).\end{align}
Consequently, by applying \e{4.54}, we obtain case (b) of the theorem after choosing a suitable coordinate system of $\mathbb E^4$.

\vskip.05in
{\it Case} (iii): $\mu=0$. This case reduces to case (ii).

The converse can be verified by straight-forward computation.
\end{proof}

Recall that an isometric immersion of a Riemannian $n$-manifold into a Euclidean $m$-space is called {\it rigid} if the isometric immersion is unique up to isometries of $\mathbb E^m$. 

For ideal hypersurfaces with two distinct principal curvatures in $\mathbb E^4$, we have the following rigidity theorem.

\begin{theorem} \label{T:4.2} Every ideal hypersurface with two distinct principal curvatures  in $\mathbb E^4$ is rigid.
\end{theorem} 
\begin{proof} From the proof of Theorem \ref{T:4.1}, we know that the second fundamental form of each ideal hypersurface  in $\mathbb E^4$ with two distinct principal curvatures depends only on the metric tensor of the ideal hypersurface. Consequently, the fundamental theorem of submanifolds implies that the ideal immersion is rigid (cf. \cite{c93,book,K}).
\end{proof}

\section{Rigidity and non-rigidity of ideal hypersurfaces with three distinct principal curvatures}

First, we give the following rigidity result.

\begin{proposition} \label{P:5.1} Every non-minimal ideal hypersurface  in $\mathbb E^4$ with  three distinct principal curvatures is rigid. \end{proposition}
\begin{proof} Assume that $M$ is a non-minimal ideal hypersurface with three distinct principal curvatures. Then it follows from Theorem \ref{T:2.1} that the three principal curvatures are $\lambda,\mu,\lambda+\mu$ for some functions $\lambda$ and $\mu$ satisfying $\lambda+\mu\ne 0$.  

Since $\lambda,\mu,\lambda+\mu$ are mutually distinct, both principal curvatures $\lambda$ and $\mu$ are nonzero. Therefore, all of the three principal curvatures must be nonzero. Hence, $M$ has type number three. Consequently, the ideal hypersurface $M$ must be rigid (cf. for instance, \cite[page 46]{K}).
\end{proof}

In view of Theorem \ref{T:4.2} and Proposition \ref{P:5.1}, we provide the following explicit examples which illustrate that minimal ideal hypersurface with three distinct principal curvatures  in $\mathbb E^4$ are not rigid in general.

\begin{example} \label{E:5.1} {\rm Let $M_1$ be the catenoid in a Euclidean 3-space $\mathbb E^3$ defined by 
\begin{align}\label{5.1} \psi_1(s,t)=\big(\! \cosh s \cos t,\cosh s\sin t,s\big)\end{align}
for $-\sinh^{-1} (1)< s<\sinh^{-1} (1)$ and $0<t<2\pi$. 
Let $M_2$ be the helicoid given by
\begin{align}\label{5.2} \psi_2(u,v)=\big(u\cos v,u \sin v,v\big)\end{align}
for $-1<u<1$ and $0<v<2\pi$. It is well-known that both the catenoid and the helicoid are minimal in $\mathbb E^3$.

Consider the map $\phi: M_1\to M_2$ defined by
\begin{align} \phi\big(\big(\! \cosh s \cos t,\cosh s\sin t,s\big)\big)= \big(\! \sinh s\cos t,\sinh s \sin t,t \big).\end{align}
It is direct to show that $\phi$ is a one-to-one isometry (cf. \cite[pages 146-147]{M}). Thus, $\psi_1$ and $\phi\circ \psi_1$ are two non-congruent isometric immersions of a Riemannian 2-manifold, say $N$, into the Euclidean 3-space $\mathbb E^3$.

If we put \begin{align} &L_1:N\times {\bf R}\to \mathbb E^4; (s,t,x)\mapsto (\cosh s \cos t,\cosh s\sin t,s,x),
\\& L_2:N\times {\bf R}\to \mathbb E^4; (s,t,x)\mapsto (\sinh s\cos t,\sinh s \sin t,t,x)\end{align}
 Then $L_1$ and $L_2$ are two non-congruent ideal immersions of the Riemannian 3-manifold $N\times {\bf R}$ into $\mathbb E^4$.
Clearly, both $L_1$ and $L_2$ have three distinct principal curvatures. 
}\end{example}

The following result is an immediate consequence of Example \ref{E:5.1}.

\begin{proposition} \label{P:5.2} There exist minimal ideal hypersurfaces in $\mathbb E^4$ with three distinct principal curvatures which are non-rigid. \end{proposition}

Finally, we give the following non-rigidity result.

\begin{proposition} \label{P:5.3} For any integer $n\geq 3$, there exist  ideal hypersurfaces in a Euclidean space $\mathbb E^{n+1}$ which are not rigid. \end{proposition}
\begin{proof} The simplest examples of such ideal hypersurfaces in ${\mathbb E}^{n+1}$ are the following two isometric immersions of $M=N\times {\mathbb E}^{n-2}$ into $\mathbb E^{n+1}$:
\begin{align} &L_1:N\times {\mathbb E}^{n-2} \ni (s,t,{\bf x})\mapsto (\cosh s \cos t,\cosh s\sin t,s,{\bf x})\in {\mathbb E}^{n+1},
\\& L_2:N\times {\mathbb E}^{n-2}\ni (s,t,{\bf x})\mapsto (\sinh s\cos t,\sinh s \sin t,t,{\bf x})\in {\mathbb E}^{n+1},\end{align}
where $N$ is defined in Example \ref{E:5.1}.
\end{proof}

An immediate consquence of Proposition \ref{P:5.3} is the following.

\begin{corollary} For each integer $n\geq 3$, there exist Riemannian $n$-manifolds which admit more than one ideal immersions in $\mathbb E^{n+1}$.
\end{corollary}

\end{document}